\documentclass[10pt, reqn, oneside]{amsart}
\usepackage{amsmath,amscd,amsthm,amsxtra,amssymb}
\usepackage{epsfig, graphics, color,colortbl}
\usepackage{amssymb,latexsym}
\usepackage{mathrsfs}
\usepackage{breqn}
\usepackage{booktabs}
\usepackage{mathtools}
\usepackage{tikz, tikz-3dplot, pgfplots}
\usepackage{tkz-graph}
\usetikzlibrary[positioning,patterns]
\usepackage{multicol}
\usepackage[poly,all]{xy}
\usepackage{marginnote}
\usepackage{xspace}
\usepackage{yfonts}
\usepackage{enumerate}
\allowdisplaybreaks[3]
\numberwithin{equation}{section}
\usepackage{microtype}  
\allowdisplaybreaks[4]  

\usepackage[margin=1in]{geometry}  

\usepackage{url}  
\usepackage{xcolor}  
\usepackage[colorlinks=true,  
            linkcolor=blue,    
            citecolor=green,   
            urlcolor=magenta]  
           {hyperref}         

\DeclareUrlCommand{\citeurl}{\urlstyle{same}}  

\theoremstyle{plain}
\newtheorem{theorem}{Theorem}[section]
\newtheorem{lemma}[theorem]{Lemma}
\newtheorem{remark}[theorem]{Remark}

\newtheorem{definition}[theorem]{Definition}
\newtheorem{proposition}[theorem]{Proposition}

\newtheorem{conjecture}{Conjecture}
\newtheorem*{theorem*}{Theorem}

\newcommand{\N}{\mathbb{N}}
\newcommand{\Z}{\mathbb{Z}}
\newcommand{\C}{\mathbb{C}}
\newcommand{\g}{\mathfrak{g}}

\newcommand{\h}{\dot{\mathfrak{h}}}
\newcommand{\ag}{\hat{\mathfrak{g}}} 

\newcommand{\ah}{\mathfrak{h}}

\newcommand{\qag}{\mathcal{U}_q(\widehat{\mathfrak{g}})}

\newcommand{\A}{\mathbb{A}}
\newcommand{\s}{\mathbb{S}}
\newcommand{\U}{\mathcal{U}}
\renewcommand{\P}{\mathcal{P}}

\begin{document}

		\title[Imaginary Whittaker modules]{Whittaker constructions for quantum affine algebras} 
\author{Vyacheslav Futorny, Santanu Tantubay}
\address{Shenzhen International Center for Mathematics, Southern University of Science and Technology, China}
\email{vfutorny@gmail.com}
\email{1mathsantanu@gmail.com}
	\maketitle
	\begin{abstract}
The goals of the paper are 3-fold. First, 
	  we revisit the construction of imaginary Whittaker modules over untwisted affine Kac-Moody Lie algebras. These modules are obtained using
    the  parabolic induction from irreducible Whittaker modules over the associated Heisenberg Lie algebras.   We show that the infinite support condition for Whittaker functions on Heisenberg Lie algebras is essential for irreducibility: when the support is finite the modules becomes reducible, yielding infinite chains of submodules. We establish the irreducibility criterion for the induced modules over affine Lie algebras and construct a large family of such modules. In particular, we obtain a class of irreducible modules on which the derivation acts neither semisimply nor freely. Second,  we consider quantum analogs of imaginary Whittaker modules and establish irreducibility for a family of such modules.  
      Finally, we prove the irreducibility of a certain class of modules over $\mathcal{U}_q(A_1^{(1)})$, which are not quantum deformations  of irreducible modules for the affine Kac-Moody Lie algebra $A_1^{(1)}$.
      Our results can be potentially extended to all types of untwisted quantum affine algebras, providing a pathway toward their classification.

	\end{abstract}
    
\section{Introduction} Whittaker modules play an important role in the 
 representation theory of Lie (super)algebras as the simplest family of non-weight representations, cf. \cite{REB81}  for the classification problem of irreducible modules for $\mathfrak{sl}_2$. The theory originated in Kostant's fundamental paper \cite{BK78}, where the correspondence was established between the central characters of $\mathcal{U}(\mathfrak{l})$ and irreducible Whittaker modules for $\mathfrak{l}$ with a regular Whittaker function on a nilpotent radical for an arbitrary semisimple Lie algebra $\mathfrak{l}$.   
Whittaker modules were studied for various families of algebras, cf.  \cite{MO05}, \cite{BO09}, \cite{ALZ16}, \cite{LPX19}, \cite{CJ20} \cite{CH21}, \cite{S00}, \cite{XGZ21}, \cite{ZL22}, \cite{CLLW24}, \cite{FS25} and references therein.

Every closed partition of the root system of an affine Kac-Moody Lie algebra defines a triangular decomposition of this algebra, these partitions form finitely many affine Weyl group orbits (cf. \cite{JK85}, \cite{JK90}, \cite{F90}, \cite{F92}). For any such partition of the root system and any $1$-dimensional representation of the positive part of the corresponding  triangular decomposition one defines a Whittaker module for the affine Lie algebra. The case of  the \emph{standard} partition of the root system
was considered in
 \cite{ALZ16}, \cite{CLLW24}, where the authors classified all irreducible Whittaker modules over $A_1^{(1)}$ and $A_N^{(1)}$ respectively. 
 
 Whittaker construction can be also extended to parabolic subalgebras by considering the parabolic induction from Whittaker modules over the Levi part of a given parabolic subalgebra. In particular, the \emph{natural} partition of the root system leads to parabolic subalgebras, whose Levi part contains the Heisenberg subalgebra. In the contrast to the case of parabolic subalgebras coming from the standard partition of the root system, the induced modules here are not smooth. 
  A family of irreducible representations of untwisted affine Kac-Moody algebras obtained by a parabolic induction from Whittaker modules over the Heisenberg subalgebra was defined  in \cite{C08}. 
   A more general construction was considered for all affine Kac-Moody Lie algebras in \cite{CF23}.

  In the first part of this paper we rewrite and complete the proof of the main result in \cite{C08} starting with a broader class of irreducible representations of the generalized Heisenberg Lie algebra. The resulting irreducible representations are imaginary (corresponding to the natural triangular decomposition) Whittaker modules for untwisted affine Kac-Moody Lie algebras.
  
  Let $H^\prime\subset \ag$ be the generalized Heisenberg subalgebra of $\ag$ and  $H=H^\prime\oplus \C d$
 its  extension  by the derivation.

To bridge the representation theory of the Heisenberg subalgebra with that of affine Lie algebras, we employ the technique of the parabolic induction. For any module $V$ over the subalgebra $H\oplus \h$, we denote by $\widehat{V}$ the corresponding induced $\ag$-module defined by
\[\widehat{V}=\U(\ag)\otimes _{\U(\P)}V,\]
where  $\P=H\oplus \h\oplus (\g_+\otimes \C[t^{\pm 1}])$ is the relevant parabolic subalgebra.

Given $\zeta\in \text{Hom}({H}^\prime(+),\C)$ and $\lambda\in (\h\oplus \C c)^*$ such that $\lambda(c)=a\in \C^*$, we construct three type of $H\oplus \h$-modules $M_{\zeta,\lambda}$ (when supp$(\zeta)=\infty$), $ N_{\zeta,\lambda}$ (when supp$(\zeta)<\infty$), and $K_{\zeta, \lambda, b}$ ($b\in \C$) in Section \ref{ciwm}. Similarly, we construct an $H^\prime\oplus \h$-module $M^\prime_{\zeta,\lambda}$.
  
  Our first main result (cf. Theorem \ref{classical}) shows that the functor of parabolic induction from Whittaker modules for the Heisenberg Lie algebra preserves irreducibility. Moreover, the same holds for  the parabolic induction from Whittaker modules for the Heisenberg Lie subalgebra without the derivation to the derived algebra of a non-twisted affine Kac-Moody algebra (cf. Theorem \ref{da}). Namely we have

\begin{theorem*}[A]\label{A}
Let $\zeta\in \text{Hom}({H}^\prime(+),\C)$, $\lambda\in (\h\oplus \C c)^*$ with $\lambda(c)\neq 0$ and $b\in \C$.
\begin{itemize}
\item Let $S\in \{M,N\}$.  Assume that $H\oplus \h$-module $S_{\zeta,\lambda}$ (resp. $K_{\zeta, \lambda, b}$) is irreducible. Then the  induced $\ag$-module $\widehat{S}_{\zeta,\lambda}$ (resp. $\widehat{K}_{\zeta, \lambda, b}$) is irreducible.
\item  Let $\tilde{\mathfrak{g}}=[\ag,\ag]$ be the derived algebra of $\ag$. Assume  that  $M^\prime_{\zeta,\lambda}$ is irreducible $H^\prime\oplus \h$-module, then the induced $\tilde{\mathfrak{g}}$-module $\widetilde{M}^\prime_{\zeta,\lambda}$ is irreducible.
 \end{itemize}
\end{theorem*}

     


Further we proceed with the construction of Whittaker modules  ${M}^q_{\zeta,\lambda}$  and  $K^q_{\zeta, \lambda, t}$ for quantum Heisenberg Lie algebras with both free and semisimple  action of $D$ respectively, and establish their irreducibility. Here $(\zeta, a, t, \lambda)\in (\text{Hom}_{\text{alg}}(\mathcal{H}_q^\prime(+), \C(q^{\frac{1}{2}}
 )), \C^\times, \C,  \dot{P})$. Using the imaginary Poincare-Birkhoff-Witt basis from \cite{CFM15} we construct a class of modules for quantum affine algebras (cf. Equation \ref{qam}) induced from these Whittaker modules.

We propose a conjecture on the irreducibility of   $\widehat{M}^q_{\zeta,\lambda}$  and $\widehat{K}_{\zeta, \lambda, t}^q$ (cf. Conjecture \ref{conjecture}).
 We prove it for induced modules with a semisimple action of the derivation.
 To show this, we use the ideas from \cite{FHW15} (cf. Theorem \ref{qi}) and construct  the $\mathbb{A}$-form of the module (cf. Equation \ref{afm}) using the $\mathbb{A}$-form of $\mathcal{U}_q(\ag)$ from Theorem \ref{algaf}. Then taking the classical limits one gets the one-to-one correspondence between $\mathbb{A}$-forms of submodules of $\qag$ and   classical submodules. Applying Theorem (A), we finally obtain the irreducibility of the induced module  with a semisimple action of the derivation.  By restricting  to the quantized derived affine  subalgebra $\U_q([\ag,\ag])$ and applying Theorem (A) to the induced module $\widetilde{M}_{(\zeta,\lambda)}^{\prime,q}$ we obtain a similar irreducibility statement. We summarize the results in the following

\begin{theorem*}[B]
Let $(\zeta, a, t, \lambda)\in (\text{Hom}_{\text{alg}}(\mathcal{H}_q^\prime(+), \C(q^{\frac{1}{2}}
 )), \C^\times, \C,  \dot{P})$.
\begin{itemize}
\item If $K^q_{\zeta, \lambda, t}$ is an irreducible $\mathcal{H}_q$-module (cf.Remark \ref{rmk1} (2)), then $\widehat{K}_{\zeta, \lambda, t}^q$ is an irreducible $\U_q(\ag)$-module.
\item Is $M_{(\zeta,\lambda)}^{\prime,q}$ is an irreducible $\mathcal{H}_q^{\prime,d}$-module  (cf. Section \ref{qderaff}), then the  induced $\U_q([\ag.\ag])$-module $\widetilde{M}_{(\zeta,\lambda)}^{\prime,q}$ is  irreducible.
\end{itemize}
 \end{theorem*}



Note that the same approach can not be used to prove the irreducibility of the induced module from $M^q_{\zeta,\lambda}$, as the derivation acts freely in this case.  To tackle this problem we apply a different technique in Section 8 for $\mathcal{U}_q(A_1^{(1)})$ (cf. Proof \ref{Proof}) and prove the irreducibility of $A_1^{(1)}$-modules on which $D$ acts freely.

\begin{theorem*}[C]
    Assume that $\ag$ is the affine Kac-Moody algebra of type $A_1^{(1)}$. If $(\zeta, a, \lambda)\in (\text{Hom}_{\text{alg}}(\mathcal{H}_q^\prime(+), \C(q^{1/2})), \C^\times,\dot{P})$ and $M_{\zeta,\lambda}^q$ is an irreducible $\mathcal{H}_q$-module, then the induced $\qag$-module $\widehat{M}_{\zeta,\lambda}^q$  is irreducible.
\end{theorem*}

 The structure of the paper is as follows.
 In Section 2 we  recall the notion of Whittaker modules for any Lie algebra with a triangular decomposition.
 We also describe the Imaginary Poincare-Birkhoff-Witt basis of  quantum affine algebra  $\U_q(\ag)$ following \cite{CFM15}, and corresponding $\mathbb{A}$-form. In Section 3, we work with $\U([\ag,\ag])$ and establish a key identity (cf. Proposition \ref{Main Identity}), which  plays an important role in constructing new irreducible representations. In Section 4 we recall the results of \cite{C08} on Whittaker modules for the Heisenberg Lie algebra.  We correct a discrepancy in the above mentioned work (equation 6.12)  regarding the treatment of general elements in the proof of irreducibility for induced modules from Heisenberg to affine Kac-Moody algebras.  
  We prove in Lemma \ref{fs} that the infinite support condition for the Whittaker function is essential for the irreducibility of $M_{\zeta, a}$. As a result we are able to construct a new class of irreducible imaginary Whittaker modules for non-twisted affine Kac-Moody Lie algebras on which the derivation acts neither  semisimply nor freely.  In the case of a finite support of the Whittaker function  we get an infinite chain of submodules of the induced module. 
In Section 6, we consider Whittaker modules for the quantum Heisenberg algebra with and without the derivation $D$. Unlike in the classical case, if the action of $D$ is free then the resulting module for the quantum Heisenberg algebra  is always irreducible. This is surprisingly independent of the support of the Whittaker function (cf. Theorem \ref{qh}). In Section 7, we   construct the induced $\qag$-module $\mathbb{I}_{q}^\lambda(V_q)$ and address its irreducibility.  We also establish irreducible imaginary Whittaker modules for the quantum algebra of the derived subalgebra of a non-twisted affine Kac-Moody algebra  (cf. Theorem \ref{qafder}).
 Finally, in Section 8 we focus on the special case of $\U_q(A_1^{(1)})$. We develop a  technique to prove the irreducibility of modules with a free action of the derivation, and hence complete the proof of Theorem (C) for the type $A_1^{(1)}.$

	\section{Preliminaries}
     We denote by $\Z,\; \Z^\times,\;   \N,\; \Z_{>0},\; \C,\; \C^\times$  the sets of integers, non-zero integers, non-negative integers, positive integers, complex numbers, and non-zero complex numbers respectively. For any set $A$, we let $\N^A_{\bf{f}}:=\{f:A\rightarrow \N| \text{supp}(f)<\infty\}$.
     
     \medskip
     Given any Lie (resp. associative) algebra  $\mathfrak{a}$, we denote by $\mathfrak{a}-$ Mod the category of all $\mathfrak{a}$-modules (resp. left modules). For a Lie algebra $\mathfrak{L}$,  $\mathcal{U}(\mathfrak{L})$ denotes the universal enveloping algebra of $\mathfrak{L}$.
     \medskip
     
     Let $n$ be a fixed positive integer, $\dot{I}=\{1,2,\dots n\}$, $I=\{0,1,2,\dots n\}$. Let $q$ be an indeterminate, $\C(q^{\frac{1}{2}})$  the field of  rational functions in $q^{\frac{1}{2}}$ with complex coefficients and $\C(q^{\frac{1}{2}})^{\times}$  the set of non-zero elements in $\C(q^{\frac{1}{2}})$. For $m,r\in \Z_+$ with $r\leq m$, define  the $q$-numbers, $q$-factorials and $q$-binomials by 
 \[ [m]_q = \frac{q^m-q^{-m}}{q-q^{-1}}, \quad [m]_q^! = [m]_q[m-1]_q\cdots [1]_q, \quad	{m \brack r}_{q} = \frac{[m]^!_q}{[r]^!_q[m-r]^!_q} \]
 respectively. 
 \\ 
\subsection{Whittaker modules over a Lie algebra:}
 Let $\mathfrak{L}$ be a Lie algebra with a triangular decomposition $\mathfrak{L}=\mathfrak{L}_-\oplus \mathfrak{L}_0\oplus \mathfrak{L}_+$. Given any Lie algebra homomorphism ${\phi}: \mathfrak{L_+}\rightarrow \C$ and $M\in  \mathfrak{L}-$Mod,  a vector $v\in M$ is called a \emph{Whittaker vector of type $\phi$} if $x.v=\phi (x)v$ for all $x\in \mathfrak{L}_+$.  We denote by $Wh_\phi(M)$ the set of all Whittaker vectors of type $\phi$ in $M$. The module $M$ is called a Whittaker module of type $\phi$ if $M=\mathcal{U}(\mathfrak{L)}(Wh_\phi(M))$.
 
\subsection{Finite dimensional simple Lie algebras:} Let $\g$ be a finite dimensional complex simple Lie algebra of rank $n$ with a fixed Cartan subalgebra $\h$, $\dot{\triangle}$  the root system of $(\g,\h)$, $\dot{\Pi}=\{\alpha_i:i\in \dot{I}\}$ the set of simple roots. Denote by $\dot{\triangle}^{\pm}$  the set of positive and negative roots of $\g$ with respect to  $\dot{\Pi}$. Also let $\dot{\Pi}^\vee =\{\alpha_i^\vee:i\in \dot{I}\}$ be the sets of co-roots and $\dot{A}=(a_{ij})_{1\leq i,j\leq n}$ be the Cartan matrix for $\g$. Hence $\alpha_j(\alpha_i^\vee)=a_{ij}$ for $i,j\in \dot{I}$. Let $\dot{Q}$ be the free abelian group generated by $\dot{\Pi}$, $\dot{Q}^+$  the free monoid generated by $\dot{\Pi}^\vee$, $\{\varpi_i:i\in I \}$  the set of fundamental weights,  $\dot{P}$ (resp. $\dot{P}^+$) the free abelian group (resp. free monoid) generated by them. We denote by $(,|,)$ the symmetric invariant non-degenerate bilinear form on $\g$ (and  on $\g^*$) normalized by the condition $(\alpha|\alpha)=2$ for all short roots. For $i\in \dot{I}$ set  $d_i=\frac{(\alpha_i|\alpha_i)}{2}$.

 \subsection {Non-twisted affine Lie algebras:}
 Denote by $\ag$  the non-twisted affine Kac-Moody associated to $\g$, it has a loop realization 
 \[\ag=\g\otimes \C[t^{\pm1}]\oplus \C c\oplus \C d,\]

 where $c$ is central in $\ag$ and  $d$ is the degree derivation, with the the following brackets:
 \[[d,x\otimes t^{m_1}]=m_1x\otimes t^{m_1},\; [x\otimes t^{m_1},x\otimes t^{m_2}]=[x,y]\otimes t^{m_1+m_2}+\delta_{m_1+m_2,0}m_1(x|y)c,\]

 for $x,y\in \g,\; m,n \in \Z$. 
 
 \medskip
 
Let  $\ah=\h\oplus \C c\oplus \C d$ be the Cartan subalgebra of $\ag$ and $A=(a_{ij})_{0\leq i,j\leq n}$  the Cartan matrix of $\ag$ such that $\dot{A}$ is  obtained by removing the first row and the first column of $A$. Fix an integer $d_0$   such that $DA$ is symmetric with the diagonal matrix $D=\text{diag}\; (d_0,d_1\dots , d_n)$. 
 Let $Q=\dot{Q}\oplus \Z \delta$, and extend the form $(.|.)$ to $Q$ by setting $(\alpha|\delta)=0$ for all $\alpha\in \dot{Q}$ and $(\delta| \delta)=0$. The root system of $\ag$ is $\triangle=\triangle^{re}\cup \triangle^{im}$, where 
 $\triangle^{re}=\{\alpha+n\delta:\alpha \in \dot{\triangle},\; n\in \Z\}$ is the set of real roots and  $\triangle^{im}=\{k\delta:k\in \Z ,\; k\neq 0\}$ is the set of imaginary roots.   Denote by $\theta=\sum_{i=1}^na_i\alpha_i$  the highest positive  root of $\g$ and set $\alpha_0=\delta-\theta$. Then $\Pi=\{\alpha_i, \, i\in I\}$ is a set of simple roots of $\ag$. Let  $Q^+$ be the free monoid generated by $\Pi$,  $W$  the Weyl group generated by simple reflections $\{r_i:i\in I\}$ and $\mathfrak{B}$  the associated braid group with generators $\{T_i:i\in I\}$. Define the weight lattice as $P=\{\lambda\in \ah^*:\; \lambda(\alpha_i^\vee)\in \Z,\; \lambda(d)\in \Z\}$. \\

\subsection{Partitions of $\triangle$:}
A subset $S\subset \triangle$ is called a partition of $\triangle$ if $S\; \cup\; -S=\triangle$ and $S\; \cap -S=\emptyset$. It is called a closed partition if $S$ is closed (if $\alpha,\beta \in S$ and $\alpha+\beta \in \triangle$ then $\alpha+\beta \in S$). Given any closed partition $S$ of $\triangle$, we have a triangular decomposition of the affine Lie algebra:
\[\ag=\bigoplus_{\alpha\in-S}\ag_{\alpha}\oplus \ah \oplus \bigoplus_{\alpha\in S}\ag_{\alpha}.\]
 There are two extreme non-equivalent closed partitions of $\triangle$ known as standard and natural, whose explicit description is given by
 \[S_{\text{st}}=\{\alpha+n\delta:\alpha \in \dot{\triangle}, n>0\}\; \cup \; \dot{\triangle}_+\; \cup \;\{n\delta:n> 0\} \]
 and 
 \[S_{\text{nat}}=\{\alpha+n\delta: \alpha\in \dot{\triangle}_+,\; n\in \Z\}\;\cup\; \{n\delta:n> 0\}. \]
 For more about the properties of partitions of root systems, one can see \cite{DFD09}. 
 \begin{remark}
 \begin{enumerate}
     \item Classification of simple Whittaker modules over $\ag$ with respect to the triangular decomposition corresponding to $S_{\text{st}}$ is given in \cite{ALZ16} for $A_1^{(1)}$, in \cite{CLLW24} for $A_n^{(1)}$ and in \cite{LL24} for $B_n^{(1)}, \; C_n^{(1)}, D_n^{(1)}$. 
     \item In current paper we will construct  a large class of simple Whittaker modules over $\ag$ with respect to the triangular decomposition corresponding to $S_{\text{nat}}$. Note that $(\ag, \ag_+^{\text{st}})$ is not a Whittaker pair in the sense of \cite{BM11}.     
 \end{enumerate}
\end{remark}
 
\subsection{Drinfeld-Jimbo realizations of $\qag$:}
  The quantum affine algebra $\U_q(\ag)$ is the unital $\C(q^{1/2})$-algebra generated by 
 \[E_i,\; F_i,\; K_\alpha,\;\gamma^{\pm 1/2}, D^{\pm 1}, \; i\in I,\; \alpha \in Q \]
 subject to the relations:
 \begin{align*}     &DD^{-1}=D^{-1}D=K_{\alpha}K_{\alpha}^{-1}=K_{\alpha}^{-1}K_{\alpha}=\gamma^{1/2}\gamma^{-1/2}=\gamma^{-1/2}\gamma^{1/2}=1,\\&[\gamma^{\pm 1/2},\U_q(\ag)]=[D,K_i^{\pm 1}]=[K_i,K_j]=0,\\
 &(\gamma^{\pm 1/2})^2=K_{\delta}^{\pm 1},\\
 &E_iF_j-F_jE_i=\delta_{ij}\frac{K_i-K_i^{-1}}{q_i-q_i^{-1}},\\
 &K_{\alpha}E_iK_{\alpha}^{-1}=q^{(\alpha,\alpha_i)}E_i,\; K_{\alpha}F_iK_{\alpha}^{-1}=q^{-(\alpha,\alpha_i)}F_i,\\
 &DE_iD^{-1}=q^{\delta_{0,i}}E_i,\; DF_iD^{-1}=q^{-\delta_{0,i}}F_i,\\
 &\sum_{s=0}^{1-a_{ij}}(-1)^sE_i^{(1-a_{ij}-s)}E_jE_i^{(s)}=0=\sum_{s=0}^{1-a_{ij}}(-1)^sF_i^{(1-a_{ij}-s)}F_jF_i^{(s)}, \; i\neq j,\\
 \end{align*}
  where $q_i=q^{d_i}, \; [m]_i=[m]_{q^i},\; [m]_i!:=\prod_{k=1}^m[k]_i$ and $K_i=K_{\alpha_i}$, $E_i^{(s)}=\frac{E_i}{[s]_i!}$ and $F_i^{(s)}=\frac{F_i}{[s]_i!}$ (cf. \cite{L88} and \cite{B94}).
  
  \medskip

  Let $\U_q^+=\qag^+$ (resp. $\U_q^-=\qag^-$) be the subalgebra of $\qag$ generated by $E_i$ (respec. $F_i$), $i\in I$, and $\U_q^0$ denote the subalgebra generated by $K_i^{\pm 1} (i\in I), \;\gamma^{\pm \frac{1}{2}},$ and $D^{\pm 1}$. 
  Let $\Phi: \U_q(\ag)\rightarrow \U_q(\ag)$ be the $\C$-algebra automorphism defined by
  \begin{equation*}
  \begin{split} 
      \Phi(E_i)=F_i, \hspace{3mm} \Phi (F_i)=E_i, \hspace{3mm}  \Phi(K_{\alpha})=K_{\alpha}\\
      \Phi(D)=D, \hspace{3mm} \Phi(\gamma^{\pm \frac{1}{2}})=\gamma^{\pm \frac{1}{2}}, \hspace{3mm} \Phi(q^{\pm \frac{1}{2}})=q^{\mp \frac{1}{2}}.
      \end{split}
  \end{equation*}
  Also let $\Omega: \U_q(\ag)\rightarrow \U_q(\ag)$ be the $\C$-algebra anti-automorphism 
  given by
   \begin{equation*}
  \begin{split} 
      \Omega(E_i)=F_i, \hspace{3mm} \Omega (F_i)=E_i, \hspace{3mm}  \Omega(K_{\alpha})=K_{-\alpha}\\
      \Omega(D)=D^{-1}, \hspace{3mm} \Omega(\gamma^{\pm \frac{1}{2}})=\gamma^{\mp \frac{1}{2}}, \hspace{3mm} \Omega(q^{\pm \frac{1}{2}})=q^{\mp \frac{1}{2}}.
      \end{split}
  \end{equation*}

 \subsection{Drinfeld Realizations of $\qag$(\cite{D85}):}
 The quantum affine algebra $\U_q(\ag)$ is an associative algebra over $\C(q^{\frac{1}{2}})$ with generators 
 $x_{i,r}^{\pm 1},  h_{i,s},\; K_i^{\pm 1},\;  \gamma^{\pm1/2},\; D^{\pm 1}\; i\in \dot{I}, r,s\in \Z,\; s\neq 0$ subject to the relations:

\begin{align*}
 &DD^{-1}=D^{-1}D=K_iK_i^{-1}=\gamma^ {1/2}\gamma^{-1/2} =1, \\
 &[\gamma^{\pm1/2},\qag]=[D,K_i^{\pm 1}]=[K_i,K_j]=[K_i,h_{j,s} ]=0,\\
 &Dh_{ir}D^{-1}=q^rh_{i,r}\; \; \; Dx^{\pm1}_{ir}D^{-1}=q^rx^{\pm}_{ir},\\
 &K_ix_{j,r}^\pm K_i^{-1}=q^{\pm (\alpha_i|\alpha_j)}x_{j,r}^\pm,\\
 &[h_{ik},h_{jl}]=\delta_{k,-l}\frac{1}{k}[ka_{i,j}]_i\frac{\gamma^k-\gamma^{-k}}{q_j-q_j^{-1}},\\
 &[h_{i,k},x^{\pm}_{j,l}]=\pm \frac{1}{k}[ka_{i,j}]_i \gamma^{\pm |k|/2}x^{\pm}_{j,k+l},\\
 &x^{\pm}_{i,k+l}x_{j,l}^{\pm 1}-q^{\pm (\alpha_i,\alpha_j)}x_{jl}^{\pm 1}x^{\pm}_{i,k+l}=q^{\pm (\alpha_i,\alpha_j)}x^{\pm}_{i,k}x_{j,l+1}^{\pm 1}-x_{j,l+1}^{\pm 1}x^{\pm}_{i,k},\\
 &[x_{i,k}^+, x_{jl}^-]=\delta_{i,j}\frac{1}{q_i-q_i^{-1}}(\gamma^{(k-l)/2}\psi_{i,k+l}-\gamma^{(l-k)/2}\phi_{i,k+l}),\\
\end{align*}  

 where
 
 \begin{align*}
  &\sum_{k=0}^{\infty}\psi_{i,k}z^k=K_i \;exp((q_i-q_i^{-1})\sum_{l>0}h_{il}z^l),\\
  &\sum_{k=0}^{\infty}\phi_{i,-k}z^{-k}=K_i^{-1}\; exp(-(q_i-q_i^{-1})\sum_{l>0}h_{i,-l}z^{-l}),\\
 \end{align*}
 
 and for $i\neq j$, 
 
 \[Sym_{k_1,\dots k_1-{a_{ij}}}\sum_{r=0}^{1-a_{ij}}(-1)^r{1-a_{ij}\choose r}_ix_{ik_1}^{\pm}\cdots x_{ik_r}^\pm x_{ij}^{\pm}x_{ik_{r+1}}^\pm\dots  x_{ik_{1-a_{i,j}}}^\pm=0.\]
 
\medskip
 We consider the natural triangular decomposition of $\qag$. Let $\U_q^+(S_{nat})$ (resp. $\U_q^-(S_{nat})$) be the subalgebra of $\qag$ generated by $x_{ik}^+$ (resp. $x_{ik}^-$), $i\in \dot{I}, \;k\in \Z$, $h_{il},$ $i\in \dot{I}$  and $ l>0$ (resp. $l<0$).  Also assume $\U_q^0(S_{nat})=\U_q^0$.

 \medskip

 Assume $\U_q[\ag, \ag]$ be the $\C(q^{\frac{1}{2}})$-subalgebra of $\qag$ generated by 

 \[\{x_{i,r}^{\pm 1},  h_{i,s},\; K_i^{\pm 1},\;  \gamma^{\pm1/2},\; i\in \dot{I}, r,s\in \Z,\; s\neq 0\}\]
subject to all the relations of $\qag$ that do not involve $D^{\pm 1}$.
 
\subsection{Imaginary PBW bases for $\qag$:} 
In this subsection we  recall the \emph{imaginary PBW basis} for $\qag$ from \cite{CFM15}, which plays a crucial role in the construction of  quantum imaginary Whittaker modules. We  will use the following theorem.

\begin{theorem}[\cite{CFM15}, Theorem 3.4.7]\label{IPBW}
 Given any $f\in \N^{\Z}_{\bf f}$, set $X^+(f):=\prod_{r\in \Z}X_{\beta_{m}^+}^{f(r)}$ and $X^-(f):=\prod_{r\in \Z}X_{\beta_{m}^-}^{f(r)}$, where the products are defined using the usual ordering of $\Z$.   Also, for $g\in \N^{\N\times \dot{I}}_{\bf f}$ we define $E^{\text{im}}(g):=\prod_{(r,i)\in \N\times I}E_{(r\delta,i)}^{g(r,i)}$, $F^{\text{im}}(g):=\Omega(E^{\text{im}}(g))$. Then the set $\{X^-(f_1)F^{im}(g_2)K_{\alpha}D^r\gamma^{\frac{s}{2}}F^{im}(g_2)X^+(f_2):f_1,f_2\in \N^{\Z}_{\bf f}, g_1,g_2\in \N^{\N\times \dot{I}}, r,s\in \Z\}$ is a  basis of $\U_q(\ag)$. 
\end{theorem}

Recall the action of the Braid group $\mathfrak{B}$  generated by  $T_i,\; i\in I$ on $\ag$:
\begin{equation*}
 T_iE_i=-F_iK_i,\; T_iE_j=\sum_{r=0}^{-a_{ij}}(-1)^{r-a_{ij}}q_i^{-r}E_{i}^{(-a_{ij}-r)}E_jE_i^{(s)}\;\; \text{if} \;i\neq j   
\end{equation*}
\begin{equation*}
 T_iF_i=-K_i^{-1}E_i,\; T_iF_j=\sum_{r=0}^{-a_{ij}}(-1)^{r-a_{ij}}q_i^{r}F_{i}^{(s)}F_jF_i^{(-a_{ij}-r)}\;\; \text{if} \;i\neq j   
\end{equation*}
\begin{equation*}
    T_iK_{\beta}=K_{s_i\beta}, \; \;\beta\in Q, \;\;\; T_iD=DK_i^{-\delta_{i0}}.
\end{equation*}
Then we have $\Omega T_i=T_i\Omega$ and $\Phi T_i=T_i^{-1}\Phi$.

 \subsection{$\mathbb{A}$-form of $\qag$:}
We will use the following $\mathbb{A}$-form of $\qag$ (cf. \cite{CFM15}). Let $\mathbb{A}=\C[q^{\pm \frac{1}{2}},\frac{1}{[n]_{q_i}},\; i\in I,\; n>1]$ and let $U_{\mathbb{A}}$ be the unital $\mathbb{A}$-subalgebra of $\qag$ generated by the elements

\[x_{ir}^{\pm},\; h_{is}, \; K_i^{\pm1}, \gamma^{\pm \frac{1}{2}},\; D^{\pm 1},\;  \left [K_i\;; \;s \atop n\right],\; \left [D\;; \;s \atop n\right],\; \left [\gamma\;; \;s \atop 1\right]_i, \left [\gamma \psi_i\;; \;k,\;l \atop 1\right]\]

 where $i\in \dot{I}, \; r,s\in \Z, \;s\neq 0$ with 
 \[\left[\gamma\;; \;s \atop 1\right]_i=\frac{\gamma^s-\gamma^{-s}}{q_i-q_i^{-1}},\]

\[\left [\gamma \psi_i\;; \;k,\;l \atop 1\right]=\frac{\gamma^{\frac{k-l}{2}}\psi_{i,k+l}-\gamma^{\frac{l-k}{2}}\phi_{i,k+l}}{q_i-q_i^{-1}},\]
 
\[\left [K_i\;; \;s \atop n\right]=\prod_{r=1}^n\frac{K_iq_i^{s-r+1}-K_i^{-1}q_i^{-(s-r+1)}}{q_i^r-q_i^{-r}},\]

\[\left [D\;; \;s \atop n\right] =\prod_{r=1}^n\frac{Dq_0^{s-r+1}- D^{-1}q_0^{-(s-r+1)}}{q_0^r-q_0^{-r}},\]
where $q_0=q^{d_0}$. Let $\U_{\mathbb{A}}^+$ (resp. $\U_{\mathbb{A}}^-$) denote the subalgebra of $\U_{\mathbb{A}}$ generated by $x_{ik}^+,\; h_{i,l}$, where $i\in \dot{I}, k\in \Z$, $l\in \Z_{>0}$ (resp. $x_{ik}^-,\; h_{i,l}$, where $i\in \dot{I}, k\in \Z$, $l\in \Z_{<0}$).  Let $\U_{\mathbb{A}}^0$ be the subalgebra of $\U_{\mathbb{A}}$ generated by the elements $\gamma^{\pm \frac{1}{2}},\; K_i^{\pm1},\; \left [K_i\;; \;s \atop n\right],\; D^{\pm 1},\; \left [D\;; \;s \atop n\right],\; \left [\gamma\;; \;s \atop 1\right]_i,\; \left [\gamma \psi_i\;; \;k,\;-k \atop 1\right].$

\begin{theorem}[\cite{CFM15}, Corollary 3.4.3]\label{algaf}
    The algebra $\U_{\mathbb{A}}$ inherits the natural triangular decomposition of $\qag$. In particular, any element $u$ of $\U_{\mathbb{A}}$ can be written as an $\mathbb{A}$-linear combination of monomials of the form $u^-u^0u^+$ where $u^\pm\in \U_{\mathbb{A}}^\pm$ and $u^0\in \U^0_{\mathbb{A}}$.
\end{theorem}
 
\section{Key identities in $\mathcal{U}([\ag,\ag])$}
 Let $\tilde{\g}=[\ag, \ag]$ be the derived algebra of $\ag$, that is the universal central extensions of the loop algebra of $\g$. 
 
 First we fix a total ordering on $\dot{\Pi}$ by $\alpha_1<\alpha_2<\cdots<\alpha_n$. Then we extend it to $\dot{Q}^+$ lexicographically. This will also induce a total ordering on  $\dot{\triangle}^+$, say $\beta_1<\beta_2<\cdots < \beta_m$. Hence, if $\beta_j-\beta_i\in \dot\triangle$ for some $1\leq i<j\leq m$ will imply that $\beta_j-\beta_i\in \dot\triangle_+$. We have $\g_+=\operatorname{span}_{\C}\{X_{\beta_i}:1\leq i\leq m  \}$, $\g_-=span_{\C}\{Y_{\beta_i}:1\leq i\leq m \}$ and
$\g=span_{\C}\{X_{\beta_i},Y_{\beta_i}, h_{\alpha_j}:1\leq i\leq m,\; 1\leq j\leq n\}$ where $h_{\beta_i}=[X_{\beta_i}, Y_{\beta_i}]$ is a $\Z$-linear combination of $\{h_{\alpha_j}: 1\leq j\leq n\}$.  For any $\beta\in \dot{\triangle}_{+}$, $\{X_{\beta},Y_{\beta}, h_{\beta}\}$ is the $\mathfrak{sl}_2$-triple.


Given any $p\in \N^\Z_{\bf{f}}$ and $\beta_i\in \dot{\triangle}^+$ we define
\[Y_{\beta_{i}}^p:=\prod_{r\in \Z}(Y_{\beta_i}(r))^{p(r)}\in\mathcal{U}([\ag,\ag]),\]
where $x(r):=x\otimes t^r$ for $r\in \Z$ and $x\in \g$. The following is obvious.
\begin{lemma}
    $\mathcal{U}(\g_-\otimes \C[t^{\pm1}])=span_{\C}\{Y_{\beta_{1}}^{p_1}\cdots Y_{\beta_{m}}^{p_m}: p_i\in \N^\Z_{\bf f}, 1\leq i\leq m\}$.
\end{lemma}
Let $ht_1:\dot{Q}_+\rightarrow\N$ is given by $ht_1(\alpha)=\sum_{i=1}^nk_i$, where $\alpha=\sum_{i=1}^nk_i\alpha_i\in \dot{Q}_+$. Given any $p\in \N^{\Z}_{\bf f}$, we define the height of $p$ as 
\[\text{ht}(p):=\sum_{r\in \Z}p(r).\]

 Clearly $\mathcal{U}(\g_-\otimes \C[t^{\pm1}])$ is a weight $\h$-module with respect to the adjoint action. For a monomial of the form $Y_{\beta_{1}}^{p_1}\cdots Y_{\beta_{m}}^{p_m}$, its weight is 
\[\text{wt}(Y_{\beta_{1}}^{p_1}\cdots Y_{\beta_{m}}^{p_m})=\sum_{i=1}^mht(p_i)(-\beta_i)\in -\dot{Q}_+.\] 
Given a monomial in the above form we also say that its height is 
 \[ht(Y_{\beta_{1}}^{p_1}\cdots Y_{\beta_{m}}^{p_m})=\sum_{i=1}^m ht(p_i) ht_1(\beta_i).\]

 Then we have the following easy statements
\begin{lemma}  Let $r,s\in \Z$ and $k\in \N$. The following holds.
\begin{itemize}
\item[1.]   
    $h_{\beta}(r).Y_{\beta}^k(s)=Y_{\beta}^k(s)h_{\beta}(r)-\beta(h_\beta)kY_{\beta}^{k-1}(s)Y_{\beta}(r+s).$
\item[2.]
     If $r\neq -s$  then we have
     \[X_{\beta}(r).Y_{\beta}^k(s)=Y_{\beta}^k(s)X_{\beta}(r)+kY_{\beta}^{k-1}(s)h_{\beta}(r+s)-\frac{\beta(h_\beta)}{2}k(k-1)Y_{\beta}^{k-2}(s)Y_{\beta}(r+s).\]
\end{itemize}
\end{lemma}

\begin{proposition}\label{Main Identity}
\begin{itemize}
\item[1.]
    Let $p\in \N^\Z_{\bf f}$ and $r\notin -\text{supp}(p)$. Then we have
    \begin{equation*}
          X_\beta(r)Y_{\beta}^p=Y^p_\beta X_\beta(r)+\sum_{s\in \text{supp}(p)} \frac{\partial Y_{\beta}^p}{\partial Y_{\beta}(s)}h_{\beta}(r+s)-\frac{\beta(h_\beta)}{2}\sum_{s_1,s_2\in \text{supp}(p)}\frac{\partial^2 Y_{\beta}^p}{\partial Y_{\beta}(s_1)\partial Y_{\beta}(s_2)}Y_{\beta}(s_1+s_2+r).
    \end{equation*}

\item[2.]
    Assume $1\leq i\leq m$, $p_i,\cdots ,p_m\in \N^{\Z}_{\bf f}$ and  $r\notin \cup_{j=i}^m-\text{supp}(p_j)$. Then we have 
\begin{multline}
    X_{\beta_i}(r)Y_{\beta_i}^{p_i}\cdots Y_{\beta_m}^{p_m}=Y_{\beta_i}^{p_i}Y_{\beta_{i+1}}^{p_{i+1}}\cdots Y_{\beta_m}^{p_m}X_{\beta_i}(r)+ Y_{\beta_i}^{p_i} [X_{\beta_i}(r),Y_{\beta_{i+1}}^{p_{i+1}}\cdots Y_{\beta_m}^{p_m} ]
    \\ +\sum_{s\in \text{supp}(p_i)} \frac{\partial Y_{\beta_i}^{p_i}}{\partial Y_{\beta_i}(s)}Y_{\beta_{i+1}}^{p_{i+1}}\cdots Y_{\beta_m}^{p_m} h_{\beta_i}(r+s)\\
    -\frac{\beta_i(h_{\beta_i})}{2}\sum_{s_1,s_2\in \text{supp}(p_i)}\frac{\partial^2 Y_{\beta_i}^{p_i}}{\partial Y_{\beta_i}(s_1)\partial Y_{\beta_i}(s_2)}Y_{\beta_i}(s_1+s_2+r)Y_{\beta_{i+1}}^{p_{i+1}}\cdots Y_{\beta_m}^{p_m}    
    \\ 
   -\sum_{s\in \text{supp}(p_i)}\sum_{t=i+1}^m \frac{\partial Y_{\beta_i}^{p_i}}{\partial Y_{\beta_i}(s)} (Y_{\beta_{i+1}}^{p_{i+1}})^*_t\cdots (Y_{\beta_m}^{p_m})^*_t
    \end{multline}
    where 
    \[(Y_{\beta_{i+l}}^{p_{i+l}})^*_t=\begin{cases}
     Y_{\beta_{i+l}}^{p_{i+l}} & \text{if} \; t\neq i+l\\
    \beta_{i+l}(h_{{\beta}_i}) \sum_{s_1\in \text{supp}(p_{i+l})} \frac{\partial Y_{\beta_{i+l}}^{p_{i+l}}}{\partial Y_{\beta_{i+l}}(s_1)}  Y_{\beta_{i+l}}(r+s+s_1)  & \text{if} \; t= i+l
    \end{cases}\]
    with $1\leq l\leq m-i$.
    \end{itemize}
\end{proposition}

\begin{proof}
      Using the previous  lemma sufficiently many times we get the desired identities.
  \end{proof}

\begin{remark} 
Note that  partial derivatives above are defined  formally. Also note
that the second expression on the right-hand side of the previous proposition does not contain any elements involving imaginary roots.  
\end{remark}

\section{Whittaker modules over Heisenberg Lie algebras}
Consider the generalized Heisenberg   Lie subalgebras of $\ag$:
 
\[H^\prime=\oplus _{\alpha\in \Delta^{im}}\ag_{\alpha}\oplus \C c \;,\]

and its  extension  by the derivation:
\[ H=H^\prime\oplus \C d,\]

Let $\{\xi_i :i\in \dot{I}\}$ be an orthonormal basis of $\h$ with respect to the bilinear form $\mathbf{(}-,-\mathbf{)}$. We set 
\begin{equation}
b_{ir}:=\begin{cases}
    \frac{1}{r}\xi_i\otimes t^r & \text{if} \;r>0, \\
    \xi_i\otimes t^r &  \text{if} \;r<0
\end{cases}
\end{equation}

for $(i,r)\in \dot{I} \times \Z^\times$.  
We have:
\begin{equation}\label{hr}
[b_{ir},b_{js}]=\delta_{ij}\delta_{r+s,0}c,\;  [d,b_{ir}]=rb_{ir},    
\end{equation}
 
where $(i,r), (j,s)\in \dot{I}\times \Z^\times$.\\

Consider the  triangular decompositions of $H^\prime$ and $H$ as follows: $H^\prime=H^\prime(-)\oplus H^\prime(0)\oplus H^\prime(+)$ and $H=H(-)\oplus H(0)\oplus H(+)$,  where 
\begin{align*}
 H^\prime(\pm)=span_{\C}\{b_{ir}: (i,\pm r)\in  \dot{I}\times \Z_{>0}\},\; H^\prime(0)=\C c,\\
 H(+)= H^\prime(+),\; H(0)=\C c,\\H(-)=\C\{b_{ir},d:  (i,r)\in \dot{I}\times \Z_{<0}\}.
\end{align*}

 For any  $\zeta \in \text{Hom}_{\C }(H^\prime(+),\C)$ and $a\in \C$ one defines the following modules
 \begin{equation}\label{hvm}
   M^\prime_{\zeta,a}:=\mathcal{U}(H^{\prime})\otimes_{\mathcal{U}(H^{\prime}(+)\oplus H^\prime(0))}\C v,  
 \end{equation}

 \begin{equation}
   M_{\zeta,a}:=\mathcal{U}(H)\otimes_{\mathcal{U}(H(+)\oplus H(0))}\C v,  
 \end{equation}

where the positive parts of the Lie algebras act on $\C v$ by the Lie algebra homomorphism $\zeta$ and $c$ acts by multiplication by $a$.
 The irreducibility of these modules is as follows (cf. \cite{C08}, Proposition $6$ and Proposition $25$).

\begin{proposition}
Let $ H^\prime,H, M^\prime_{\zeta,a}, M_{\zeta,a}$ be as above and $a\in \C^\times$.
  \begin{enumerate}
      \item If  $\zeta\neq0$ then $M^\prime_{\zeta,a}$ is an irreducible $H^\prime$module, then all Whittaker $H^\prime$-modules  are of this form.
      \item If  $\text{Supp}(\zeta)=\infty$ then  $M_{\zeta,a}$ is irreducible.
  \end{enumerate}  
\end{proposition}
We will see now that the infinite support of the Whittaker function is essential for $M_{\zeta, a}$ to be irreducible.

\begin{lemma}\label{fs}
   If $a\in \C^\times$ and $\text{supp}(\zeta)<\infty$  then  $M_{\zeta,a}$ is reducible with an infinite composition series. 
\end{lemma}
\begin{proof}
   Let $w_k$ be any nonzero Whittaker vector of $M_{\zeta,a}$. Then it can be shown that \[w_{k+1}=d.w_{k}+\sum_{(i,r)\in \dot{I}\times \Z_{>0}}\frac{r\zeta(b_{i,r})}{a}b_{i,-r}.w_{k}\] is also a Whittaker vector.
   
   Now setting $w_1=1\otimes 1\in M_{\zeta,a}$ and  $M_k=\mathcal{U}(H).w_k$ we obtain an infinite chain of submodules 
   $$M_{\zeta,a}=M_0\supsetneq M_1\supsetneq M_2\supsetneq \cdots M_k\supsetneq M_{k+1}\supsetneq\cdots$$   with $M_{k}/M_{k+1}\cong M^\prime_{\zeta,a}$. Then all $M_{k}/M_{k+1}$ are irreducible as $H$-modules. 
   \end{proof}
   Denote by $\bar{w}_k$  the image of $w_k$ in $M_k/M_{k+1}$. Then the action of  the derivation is defined as follows:
   \begin{equation}\label{dera}
       d.\bar{w}_k=-\sum_{(i,r)\in \dot{I}\times \Z_{>0}}\frac{r\zeta(b_{i,r})}{a}b_{i,-r}.\bar{w}_{k}.
   \end{equation}
   
   \begin{remark}
   \begin{enumerate}
   \item We denote $M_{\zeta, a}/M_1$ by $N_{\zeta, a}$, when $\text{supp}(\zeta)<\infty$.
   
   \item The $H^\prime$-module $M^\prime_{\zeta, a}$ can be viewed as an $H$-module by setting $d.v=bv$ for some $b\in \C$ and extending the derivation action to the whole module. We denote these modules by $K_{\zeta,a, b}$.
   \end{enumerate}

   \end{remark}
 \section{Imaginary Whittaker modules}
 \subsection{Imaginary Whittaker function:}
 Every closed partition of the root system of $\ag$ yields a triangular decoposition $\ag=\ag_-\oplus \ah\oplus \ag_+$. If we choose the natural triangular decomposition of $\ag$ corresponding to $S_{nat}$, then for any Lie algebra homomorphism ${\zeta}:\ag_+\rightarrow \C$, all real root vectors belong to the kernel of this function.

 \subsection{A class of imaginary Whittaker modules:}\label{ciwm} 

 



 
  Now we consider three different classes of irreducible $H$-modules: $M_{\zeta,a}$ (when support of $\zeta$ is infinite), $N_{\zeta,a}$ (when support of $\zeta$ is finite) and $K_{\zeta,a,b}$. Also 
  given any $\lambda\in \h^*$, we assume that $\h$ acts on these spaces by the functional $\lambda$. Given the functional $\lambda\in \h$, we can extend it to $\h\oplus \C c$ by denoting by the same such that $\lambda(c)=a$ and therefore we denote the above three spaces by $M_{\zeta, \lambda,}, N_{\zeta, \lambda},\; K_{\zeta, \lambda,b}$ when they are considered as $H\oplus \h$ modules. In a similar fashion we extend the $H^\prime$-module $M^\prime_{\zeta,a}$ to $H^\prime\oplus \h$ module by the scalar action of $\h$ given by $\lambda$ and denote it by $M^\prime_{\zeta, \lambda}$.

 Let $\P=H\oplus \h\oplus (\g_+\otimes \C[t^{\pm1}])$ (resp. $\widetilde{\P}=H^\prime\oplus \h\oplus (\g_+\otimes \C[t^{\pm1}])$) be a parabolic subalgebra of $\ag$ (resp. $\tilde{\g}$). Given any  $H\oplus \h$-module (resp. $H^\prime\oplus \h$-module), we  make it into a $\P$-module (resp. $\widetilde{\P}$-module) with the trivial action of $ \g_+\otimes\C[t^{\pm 1}]$.
 
 Now given any $\P$ module (resp. $\widetilde{\P}$ module)  $V$ (resp. $U$), we get $\ag$ module (resp. $\tilde{g}$-module) via parabolic induction 
 \[Ind_{\P}^{\ag}V:=\U(\ag)\otimes _{\U(\P)}V, \quad (\text{resp.} \;Ind_{\widetilde{\P}}^{\tilde{g}} U:=\U(\tilde{g})\otimes_{\U(\widetilde{\P})} U),\]
 and we denote it by $\widehat{V}$ (resp. $\widetilde{U}$).

   \textbf{Notation:} For ${\bf p}=(p^1,\dots ,p^m)\in (\N^\Z_{\bf f})^m$, we set $Y^{\bf p}=Y_{\beta_1}^{p_1 }\cdots Y_{\beta_m}^{p_m}$.

 \begin{theorem}\label{classical}
The modules $\widehat{M}_{\zeta, \lambda}$ (when support of $\zeta$ is infinite), $\widehat{N}_{\zeta,\lambda}$ (when support of $\zeta$ is finite) and $\widehat{K}_{\zeta,\lambda,b}$  are irreducible for $\ag$.
  \end{theorem}
 
 \begin{proof}
  Let $V$ be any one of $M_{\zeta,\lambda}$ (when support of $\zeta$ is infinite) or $N_{\zeta,\lambda}$ (when support of $\zeta$ is finite) or $K_{\zeta,\lambda, b}$.
  Let $L$ be any non-zero submodule of $\widehat{V}$. It is sufficient to show that $L\cap V\neq 0$. We have that $V\cong \C[h_{\alpha_i}(r):r\in [-\infty, 0),\; 1\leq i\leq n]\otimes W$ for some subspace $W$ of $V$. Since $ \widehat{V}$ is an $\h$-weight $\ag$-module, the same holds for $L$. We will proceed by induction on the height of weight elements of $L$, where height of a weight element is determined by the  height of its weight. Let $v\in L$ be a non-zero weight element of height $p(>1)\in\N$. We will show that there exists a non-zero weight element in $L$ of height strictly smaller than $p$. Write $v$ in the form 

  \begin{equation*}
   v=\sum_{{\bf p}\in I_v}Y_{\beta_1}^{p_1}Y_{\beta_2}^{p_2}\cdots Y_{\beta_m}^{p_m}\otimes v_{\bf p},   
  \end{equation*}
   where $I_v:=\{{\bf p}^1,\cdots {\bf p}^k\}$ is a finite subset of $(\N^\Z_{\bf{f}})^m$ with ${\bf p}^j=(p_1^j,\cdots ,p_m^j)$ with $1\leq j\leq k$. 
  Note that for each $1\leq j\leq k$,  there exists $\nu_j\in \N$ such that $v_{{\bf p}^j}\in \C[h_{\alpha_i}(r): r>-\nu_j]\otimes W$.
Assume $i$ to be the minimal index such that $p_i\neq 0$ for some $(0,\dots ,p_i,\dots p_m)\in I_v$.  Now, let us partition $I_v$ into disjoint unions of two sets $I_v(1)$ and $I_v(2)$, where $I_v(1)$ (resp. $I_v(2)$)
is the collection of all elements ${\bf p}\in I_{v}$ having the i-th coordinate  $p_i\neq 0$ (resp. $p_i= 0$).
 Now we choose $r\in \Z_{<0}$ such that  $r+s< -\text{max}\{\nu_j:1\leq j\leq k\}$ for all $s\in \text{supp}(p_i^j)$ and all ${\bf p}^j\in I_v(1)$.
 \\
 
 Then we have 

 \begin{multline*}
  X_{\beta_i}(r). v:=\sum_{\substack{{s_1\in \text{supp({\bf p})}}\\ {{\bf p}\in I_v(1) }}} \frac{\partial Y^{\bf p}}{\partial Y_{\beta_i}(s_1)}h_{\beta_i}(r+s_1)\otimes v_{\bf p}-
  \sum_{\substack{s_1\in \text{supp}({\bf p})\\{\bf p} \in I_v(1) }} \sum_{t=i+1}^m  \frac{\partial Y_{\beta_i}^{p_i}}{\partial Y_{\beta_i}(s)} (Y_{\beta_{i+1}}^{p_{i+1}})^*_t\cdots (Y_{\beta_m}^{p_m})^*_t \otimes v_{\bf{p}} \\
  -\frac{\beta_i(h_{\beta_i})}{2}\sum_{\substack{s_1,s_2\in \text{supp}(p_i)\\ {\bf p} \in I_v(1)}} \frac{\partial^2 ( Y_{\beta_i}^{p_i})}{\partial (Y_{\beta_i}(s_1))\partial (Y_{\beta_i}(s_2))}Y_{\beta_i}(r+s_1+s_2) Y_{\beta_{i+1}}^{p_{i+1}}\cdots Y_{\beta_m}^{p_m}\otimes v_{{\bf p}}+
  \\
 \sum_{{\bf p} \in I_v(1)} Y_{\beta_i}^{p_i}[X_{\beta_i}(r), Y_{\beta_{i+1}}^{p_{i+1}}\cdots Y_{\beta_m}^{p_m}]\otimes v_{{\bf p}}+ \sum_{{\bf p}\in I_v(2) }X_{\beta_i}(r)Y^{\bf p}\otimes v_{\bf p}
  \in K.
  \end{multline*}
  It is easy to see that $X_{\beta_i}(r).v$ is an element of height at most $ht (v)-ht(\beta_i)$. Suppose $X_{\beta_i}(r).v=0$. Then 
  $$ \sum_{\substack{{s_1\in \text{supp({\bf p})}}\\ {{\bf p}\in I_v(1) }}} \frac{\partial Y^{\bf p}}{\partial Y_{\beta_i}(s_1)}h_{\beta_i}(r+s_1)\otimes v_{\bf p}=0,$$
   which is a contradiction. 
  \end{proof}
  
  Note that  Theorem \ref{classical} implies in particular \cite[Theorem $50$]{C08}, if we choose $V=M_{\zeta,\lambda}$ with infinite support of $\zeta$, which is again a particular case of \cite{CF23}.

\begin{remark}
Note that for $V=K_{\zeta,\lambda,b}$ the imaginary Whittaker module $\widehat{V}$ is  weight.
\end{remark}

  
\begin{theorem}\label{da} 
     $ \widetilde{M}^\prime_{\zeta,\lambda}$ is irreducible for $\tilde{\g}$.
\end{theorem}
  The proof of Theorem \ref{da} follows a similar height-reduction argument to that of Theorem \ref{classical}. 
  
\section{Whittaker modules over Quantum Heisenberg Algebras}

 In this section we discuss Whittaker modules over quantum Heisenberg algebras. 
 Let  $\mathcal{H}_q^\prime$ be  the subalgebra of $\U_q(\ag)$ generated by  $h_{i,r}, \gamma^{\pm 1/2}:i\in \dot{I} , r\in \Z^\times$.  
 Following  \cite[Section 2.3]{BCFK22} or  \cite[Section 2]{FHW15}, we can choose  new set of generators $\{h_{i,r}^{\prime}:(i,r)\in \dot{I}\times \Z^\times\}$ of $\mathcal{H}_q^\prime$ such that  the following relations hold:
 \begin{equation}\label{qhr}
   [h_{ir}^\prime,h_{j,s}^\prime]=\delta_{ij}\delta_{r+s,0}\frac{\gamma^r-\gamma^{-r}}{q-q^{-1}}.  
 \end{equation}
 
 Define $\mathcal{H}_q^\prime({\pm})$ to be generated by $h_{i,\pm r}, \, (i,\pm r)\in \dot{I}\times \Z_{>0}$.  Let $\mathcal{H}_q$ be the subalgebra of $\qag$ generated by $\mathcal{H}_q^\prime$ and $D^{\pm1}$. Note that $Dh_{ir}^\prime D^{-1}=q^rh_{i,r}^\prime$. Set $\mathcal{H}_q(\pm)=\mathcal{H}^\prime_q(\pm)$ and let $\mathcal{H}_q(0)$ be the subalgebra generated by $\gamma^{\pm 1/2},D^{\pm 1}$. For the sake of simplicity we will denote $h^\prime_{ir}$ by $h_{i,r}$.

 \begin{definition}
   Assume $0\neq \zeta \in Hom_{\text{alg}}(\mathcal{H}_q^\prime(+), \C(q^{1/2}))$ and let $V$ be an $\mathcal{H}_q^\prime$-module. We say that a non-zero $v\in V$ is a Whittaker vector of type $\zeta$ if $y.v=\zeta(y)v$ for all $y\in \mathcal{H}_q^\prime(+)$. The module $V$ is called a Whittaker module of type $\zeta$ if there exists a  Whittaker vector $v$ of type $\zeta$ such that $V=\mathcal{H}_q^\prime.v$. Similarly one defines  Whittaker modules over $\mathcal{H}_q$.
 \end{definition}
 
 Now we will explicitly construct Whittaker modules over quantum Heisenberg algebra.
  Let $\mathcal{H}_q^\prime(\mathfrak{b})$ be generated by $\mathcal{H}_q^\prime(+)$ and $ \gamma^{\pm 1/2}$. It is easy to see that this is a maximal abelian subalgebra of $\mathcal{H}_q^\prime$. 
  For any $(\zeta,a) \in (Hom_{\text{alg}}(\mathcal{H}_q^\prime(+), \C(q^{1/2})), \C^\times)$,  we let $\C_{(\zeta,a)}(q^{1/2})v=\C(q^{1/2})v$ to be the one-dimensional $\mathcal{H}_q^\prime(\mathfrak{b})$-module with the action given by
  \begin{equation}\label{a}
       y.v=\zeta(y)v,\; \gamma^{1/2}.v=q^{a/2}.v
  \end{equation}
   for all $y\in \mathcal{H}_q^\prime(+)$.
Now we consider the induced $\mathcal{H}_q^\prime$-module 
\begin{equation}\label{b}
M_{\zeta,a}^{(\prime,q)}=\mathcal{H}_q^\prime\otimes_{\mathcal{H}_q^\prime(\mathfrak{b})}\C_{(\zeta,a)}(q^{1/2})v 
\end{equation} 

and the induced $\mathcal{H}_q$-module 

\begin{equation}\label{d}
M_{\zeta,a}^q =\mathcal{H}_q \otimes_{\mathcal{H}_q (\mathfrak{b})}\C_{(\zeta,a)}(q^{1/2})v.  
\end{equation} 
 Clearly $M_{\zeta,a}^q$ is Whittaker $\mathcal{H}_q$-module  and $M^{(\prime,q)}_{\zeta, a}$ is a canonical $\mathcal{H}_q^\prime$-submodule of  $M_{\zeta,a}^q$. 

\medskip

Fix a lexicographic order (say $'\leq'$) on $\N^{\dot{I}\times \Z_{>0}}_{\bf f}$. Given any $p\in \N^{\dot{I}\times \Z_{>0}}_{\bf f}$ with $p(i,r)=p_{ir}$ and  $a\in \C^\times$, we set the following:
  
\begin{enumerate}
    \item \[p!=\prod_{(i,r)\in \dot{I}\times \Z_{>0}}p_{ir}!,\]
    \item \[h_-^p=\prod_{(i,r)\in \dot{I}\times \Z_{>0}}h_{i,-r}^{p_{ir}},\]
    \item \[(h_+-\zeta)^p=\prod_{(i,r)\in \dot{I}\times \Z_{>0}}(h_{i,r}-\zeta(h_{ir}))^{p_{ir}},\]
    \item \[(h_--\eta)^p=\prod_{(i,r)\in \dot{I}\times \Z_{>0}}(h_{i,-r}-\eta(h_{i,-r}))^{p_{ir}},\]
    \item  \[[a]^p=\prod_{(i,r)\in \dot{I}\times \Z_{>0}}[ra]^{p_{ir}}, \]
\end{enumerate}
 for any $\eta\in Hom_{\text{alg}}(\mathcal{H}_q^\prime(-),\C(q^{1/2}))$.
 
 \medskip
To establish the irreducibility of the induced modules, we first investigate the action of the polynomial operators $(h_+-\zeta)^p$ on the basis elements $h_-^pv$. The following lemma shows that these operators act as detectors for the PBW basis, yielding a non-zero scalar when the partition matches and annihilating the vector when the operator's order exceeds the basis element in the lexicographic ordering.

\begin{lemma}\label{pf}
 The following equalities hold in $M^{(\prime,q)}_{\zeta,a}$:
 \begin{enumerate}
     \item If $a\in \C^\times$ then $(h_+-\zeta)^p h_-^pv=p![a]_q^p v$.
     \item If $a\in \C^\times$ and $p_1,p_2\in \N^{\dot{I}\times \Z_{>0}}_{\bf f} $ with $p_1<p_2$, then $(h_+-\zeta)^{p_2} h_-^{p_1}v=0$. 
     \item If $a=0$ then $h_{ir}h_-^pv=\zeta(h_{ir})h_-^pv$.     
 \end{enumerate}
 \begin{proof}
      
(1) We proceed by induction on the total degree of $p$.  
Applying the operator $(h_{i,r} - \zeta(h_{ir}))$ to the basis element $h_{i,-r}v$, and noting that $\gamma^r v = q^{ra}v$, we have:
\[ (h_{i,r} - \zeta(h_{ir})) h_{i,-r}v = ([h_{i,r}, h_{i,-r}] + h_{i,-r}h_{i,r} - \zeta(h_{ir})h_{i,-r})v. \]
Since $h_{i,r}v = \zeta(h_{ir})v$, the last two terms cancel, leaving $[ra]_q v$. Extending this to the product $h_-^p$ using equation \ref{qhr} and the fact that generators for different $(i, r)$ commute, we obtain the factor $p! [a]_q^p v$.

(2) Let $p_1 < p_2$ in the lexicographic order. There exists a pair $(i, r)$ such that the exponent of the operator $(h_{i,r} - \zeta(h_{ir}))$ in $(h_+ - \zeta)^{p_2}$ is strictly greater than the exponent of $h_{i,-r}$ in the basis element $h_-^{p_1}v$, while all higher-order terms match. When we commute the positive generators through, the "surplus" of operators $(h_{i,r} - \zeta(h_{ir}))$ eventually acts directly on the Whittaker vector $v$. By the definition of a Whittaker vector, $(h_{ir} - \zeta(h_{ir}))v = 0$, thus annihilating the entire expression.

(3) If $a = 0$, the central charge $\gamma$ acts as the identity since $q^0 = 1$. The commutator in \ref{qhr} vanishes because $[r \cdot 0]_q = 0$. Consequently, the algebra $\mathcal{H}_q^\prime$ becomes commutative. The positive generators $h_{ir}$ can be moved past the negative generators $h_-^p$ to act directly on $v$, yielding the eigenvalue $\zeta(h_{ir})$.
\end{proof}
  
\end{lemma}

  \begin{proposition}\label{ired1}
     Let $a\in \C^\times$. Then $M^{(\prime,q)} _{\zeta,a}$ is the unique (up to an isomorphism) irreducible Whittaker $\mathcal{H}_q^{\prime}$-module  of type $(\zeta,a)$, where $\gamma$ acts by $q^a$. 
  \end{proposition} 
  \begin{proof}
      Any non-zero element of $M^{(\prime,q)}_{\zeta,a}$ is of the form
      \[w=\lambda_1 h_-^{p_1}v+\lambda_2h_-^{p_2}v+\cdots \lambda_k h_{-}^{p_k}v,\]
      where $\lambda_j \in \C (q^{\frac{1}{2}})^\times$ and $p_j\in \N^{\dot{I}\times \Z_{>0}}$ for $1\leq j\leq k$. Any non-zero submodule must contain an element of the above form. We assume without loss of generality that $p_1$ is maximal with respect to the ordering $(\N_{\bf f}^{\dot{I}\times \Z}, \leq)$. Then applying Lemma \ref{pf} (1) and (2), we get that $(h_+-\zeta)^{p_1}.w=\lambda_1 p_1![a]^{p_1}v$ belongs to the submodule. Hence $M^{(\prime,q)}_{\zeta,a}$ must be irreducible $\mathcal{H}_q^\prime$-module.             
  \end{proof}

  \begin{remark}\label{rmk1}
  \begin{enumerate}
  \item The isomorphism  in   \cite[Proposition 2.1]{FHW15} gives a  characterization of all  simple Whittaker modules over the quantized infinite rank Weyl algebra with non-zero central charge. 
      \item  We note that for any given $t\in \Z$,  $M^{\prime,q}_{\zeta, a}$ becomes an irreducible $\mathcal{H}_q$-module by defining $D.v=q^{t}v$ and then scaling it to the whole module. Denote this module by $K^q_{\zeta, a,t}$. 
  \end{enumerate}
   \end{remark}

The following is clear. 

 \begin{lemma}
     The set $\{D^mh_-^pv: m\in \Z,\; p\in \N^{\dot{I}\times \Z_{_>0}}\}$ is a $\C(q^{1/2})$-basis of $M^q_{\zeta, a}$.
 \end{lemma}

   For any $m\in \Z$, and $p\in \N^{\dot{I}\times \Z_{>0}}_{\bf f}$ set:
  \[(h_+-q^{-m}\zeta)^p=\prod_{(i,r)\in \dot{I}\times \Z_{>0}}(h_{i,r}-q^{-mr}\zeta(h_{ir}))^{p_{ir}}.\] 
  \begin{lemma}\begin{enumerate}
      
      \item  $(h_+-q^{-m}\zeta)^p. D^mh_-^pv= p![a]^pD^m v.$  in $M_{\zeta,a}^q$. 
      \item  $(h_+-q^{-m}\zeta)^{p_1}. D^mh_-^{p_2}v=0$ for $p_1>p_2$ in  $\N^{\dot{I}\times \Z}_{\bf f}$.
      \end{enumerate}
 \end{lemma}

 \begin{proof}
     We use commutation relations of $\mathcal{H}_q$ and property of the Whittaker vector. 
 \end{proof} 
  
 \begin{theorem} \label{qh}
     Assume $a\in \C^\times$. Then the $\mathcal{H}_q$-module $M_{\zeta,a}^q$ is irreducible. 
 \end{theorem}
 
 \begin{proof}
 Suppose $N$ is a non-zero submodule of $M^q_{\zeta,a}$. Any non-zero element of $M^q_{\zeta,a}$ is of the form
 \[\sum_{\lambda_{(m_j,p_j)}\in F_v} \lambda_{(m_j,p_j)}D^{m_j}h_-^{p_j}.v,\]
where $F_v$ is a finite subset of $\C(q^{\frac{1}{2}})^\times$. We assume that $N$ contains this element.
 
    Now applying the proper $D^m$ in this element, we will have 
    \[(u+\sum \lambda^\prime_{(m_j,p_j)}D^{m_j}h_-^{p_j}).v\in N,\] where $u(\neq 0)\in {M}_{\zeta,a}^{(\prime,q)}$, $(m_j,p_j)\in \Z_{>0}\times \N^{\dot{I}\times \Z_{>0}}_{\bf f}$ and $\lambda^\prime_{(m_j,p_j)}\in \C(q^{1/2})^\times$. Now by Proposition \ref{ired1}, there exists $w\in \mathcal{H}_q^{\prime}$ such that $w.u=v$ and therefore $v+\sum \lambda^\prime_{(m_j,p_j)}w.(D^{m_j}h_-^{p_j}.v)\in N$, which gives $v+\sum \lambda^{\prime \prime}_{(m_k,p_k^\prime)}(D^{m_k}h_-^{p^{\prime}_k}.v)\in N$. Now we know that $\zeta \neq 0$, so there exists $(i,r)\in \dot{I}\times \Z_{>0}$ such that $\zeta(h_{ir})\neq 0.$ Now we see that
 
 \[\prod_{k}(h_{ir}-q^{-m_kr}\zeta(h_{ir}))^{p_k^\prime(ir)+1}.(v+\sum \lambda^{\prime \prime}_{(m_k,p_k)}(D^{m_k}h_-^{p^{\prime}_k}.v))\]\[=\zeta(h_{ir})\prod_{k}(1-q^{-m_kr})^{p_k^\prime(ir)+1}v\in N,\]
Here it is easy to see that the coefficient of $v$ is non-zero and hence $M^q_{\zeta, a}$ is irreducible for $\mathcal{H}_q$.
\end{proof}

\section{Imaginary Whittaker Modules over $\qag$}\label{aform}
In this section, we define imaginary Whittaker modules over $\qag$ and consider their irreducibility.

\medskip
Let $\U_q^d(\pm)$ be the subalgebra of $\qag$ generated by the Drinfel'd generators $x_{ik}^{\pm}\; (i\in \dot{I},\; k\in \Z)$ and $\mathcal{H}_q^d$ be the subalgebra of $\qag$ generated by $\mathcal{H}_q$ and $U_q^0$. Let $B_q^d$ be the subalgebra of $\qag$ generated by $\U_q^+$ and $\mathcal{H}_q^d$. For any $l\in \Z$, we assume $\mathcal{H}_q-\text{Mod}^{l}$ to be the subcategory of $\mathcal{H}_q-\text{Mod}$ such that $\gamma=q^{l} \;Id$ on every module of $\mathcal{H}_q-\text{Mod}^{l}$. Given any $\lambda\in \dot{P}$, we can consider the following functor

\[\mathbb{I}_q^\lambda:\mathcal{H}_q-\text{Mod}^{ l}\rightarrow \qag-\text{Mod}\]
with
\begin{equation}\label{qam}    \mathbb{I}_q^\lambda(V)=\qag\otimes _{B_q^d}V,
\end{equation}
 where on $V$, the $B_q^d$-module action is extended by $x_{ik}^+.V=0,\; K_{_i}^{\pm1}$  acts as a scalar $q^{\lambda(h_i)}$ for $i\in \dot{I}$.  

  We denote $\mathbb{I}_q^{\lambda}(M_{\zeta,a}^q)$ by $\widehat{M}_{\zeta,\lambda}^q$ and $\mathbb{I}_q^\lambda(K_{\zeta,a,t})$ by $\widehat{K}^q_{\zeta,\lambda,t}$.
  
  \begin{conjecture}\label{conjecture} 
Given any $\lambda\in \dot{P}$ and $l=a$ with  $(\zeta,a) \in (Hom_{\text{alg}}(\mathcal{H}_q^\prime(+), \C(q^{1/2})), \C^\times)$, the $\qag$-modules $\widehat{M}_{\zeta,\lambda}^q,\; \widehat{K}^q_{\zeta,\lambda,t}$ are irreducible.
  \end{conjecture}
  
  We will prove the conjecture for $V_q=K^{q}_{\zeta,a,t}$ for all nontwisted affine Kac-Moody Lie algebra $\ag$ and for $V_q=M^q_{\zeta,a}$ over $\qag$ when $\ag=A_1^{(1)}$.\\

For the rest of this section we assume  that $V_q=K^q_{\zeta,\lambda,t}$.
   Let $\mathcal{H}_{\mathbb{A}}$  be the $\mathbb{A}$-subalgebra of $\qag$ generated by $$h_{is},\;  K_{i}^{\pm1},\; \gamma^{\pm\frac{1}{2}},\; D^{\pm1},\; \left [K_i\;; \;s \atop n\right],\; \left [D\;; \;s \atop n\right],\; \left [\gamma\;; \;s \atop 1\right]_i, \left [\gamma \psi_i\;; \;k,\;l \atop 1\right], $$ 
  for $i\in \dot{I},\; r,s \in \Z,\; s\neq 0$. We see that $\mathcal{H}_q^d/<q^{\frac{1}{2}}-1>\; \cong\; \U(H+\h)$ and $\mathcal{H}_{\mathbb{A}}\otimes \C(q^{\frac{1}{2}})\cong \mathcal{H}_q^d$.  We get that $(K_{\zeta,\lambda,t}^q)^{\mathbb{A}}=\mathbb{A}\{h_-^fv:f\in \N^{\dot{I}\times \Z}_{\bf f}\}$ is an $\mathcal{H}_{\mathbb{A}}$-submodule of $K_{\zeta,a,t}^q$ and its classical limit $\overline{K}_{\zeta,\lambda,t}=(K_{\zeta,\lambda,t}^q)^{\mathbb{A}}/<q^{\frac{1}{2}}-1>(K_{\zeta,\lambda,t}^q)^{\mathbb{A}}$ is isomorphic to $K_{\zeta^{\prime},\lambda,b}$ for some $(\zeta^\prime, b)\in(\text{Hom} (H(+),\C), \C)$. Hence it is irreducible over $H\oplus \h$.

  Now we define $\mathbb{A}$-form of 
   $\widehat{K}^q_{\zeta,\lambda,t}$ to be the $\U_{\mathbb{A}}(\ag)$-submodule 
  
\begin{equation}\label{afm}
     (\widehat{K}^q_{\zeta,\lambda,t})_{\mathbb{A}}=\sum_{f\in (\N^{\Z}_{\bf f})^n}\U_{\mathbb{A} }\otimes_{B^q} h_-^{f}v, 
  \end{equation}

where $B^q$ is the $\mathbb{A}$-subalgebra generated by $\U_{\mathbb{A}}^+\; \cup\; \mathcal{H}_{\mathbb{A}}$. It is easy to see that $(\widehat{K}^q_{\zeta,\lambda,t})_{\mathbb{A}}$ is spanned by $\U_{\mathbb{A}}^-\otimes h_-^f$ as an $\mathbb{A}$-module and it is a free $\U_{\mathbb{A}}^{-}$-module; furthermore, we have $\C(q^{\frac{1}{2}})\otimes _{\mathbb{A}}(\widehat{K}^q_{\zeta,\lambda,t})_{\mathbb{A}}\cong \widehat{K}^q_{\zeta,\lambda,t}$ as $\C(q^{\frac{1}{2}})$-vector space isomorphism. 
Now we construct the classical limit of $\mathbb{I}_q^\lambda(V_q)$. Let $\mathbb{J}$ be the ideal of $\mathbb{A}$ generated by $q^{\frac{1}{2}}-1.$ Now we take $\U^\prime=\mathbb{A}/\mathbb{J}\otimes_{\mathbb{A}}\U_{\mathbb{A}}$ and $\bar{\U}=\U^\prime/K^\prime$, here $K^\prime$ is the ideal of $\U^\prime$ generated by $K_i-1, \;D-1,\; \gamma^{\frac{1}{2}}-1$, then $\bar{\U}$ is the $q^{\frac{1}{2}}=1$ limit of $\qag$ and hence $\U(\ag)\cong \bar{\U}$. Now construct $\overline{(\widehat{K}^q_{\zeta,\lambda,t})}_{\mathbb{A}}=\mathbb{A}/\mathbb{J}\otimes_{ \mathbb{A}}(\widehat{K}^q_{\zeta,\lambda,t})$, it can be shown that it is a well-defined $\bar{\U}$-module, we call it the classical limit of $(\widehat{K}^q_{\zeta,\lambda,t})_{\mathbb{A}}$. Similar to Proposition 4.6 and Proposition 4.7 of \cite{FHW15}, we get $\overline{(\widehat{K}^q_{\zeta,\lambda,t})}_{\mathbb{A}
}\cong  \widehat{\bar{V}}$, and hence we have $\mathbb{I}_q^\lambda(V_q)$ is a quantum deformation of $Ind_{P}^{\ag}( \bar{V})$. Now applying the Theorem \ref{classical}, we get one of our main results, which yields irreducible $\qag$-modules from irreducible $\mathcal{H}_q$-modules.

\begin{theorem}\label{qi}
    Let $(\zeta, a,t,\lambda)\in (Hom_{\text{alg}}(\mathcal{H}_q^\prime(+), \C(q^{1/2})), \C^\times, \C,\dot{P})$. Then $\widehat{K}^q_{\zeta,\lambda, t}$ is an irreducible $\qag$-module. 
\end{theorem}

\begin{proof} 
Let $W$ be a proper submodule of $\widehat{K}^q_{\zeta,\lambda, t}$. Then $W^{\mathbb{A}}=W\; \cap\;(\widehat{K}^q_{\zeta,\lambda, t})_{\mathbb{A}}$ is a proper submodule of $(\widehat{K}^q_{\zeta,\lambda, t})_{\mathbb{A}}$, and hence its classical limit is a proper submodule of $ \widehat{\bar{V}}$, which is a contradiction by Theorem \ref{classical}.
\end{proof}

\begin{remark}\label{rmk2}
    Note that  the classical limit can't be taking in the case of $V_q=M_{\zeta,a}^q$ as $D$ acts freely on $V_q$. 
\end{remark}
\subsection{Quantization of imaginary Whittaker modules for $\U_q([\ag,\ag])$:}\label{qderaff}
Assume $\mathcal{H}_q^{\prime,d}$ be the subalgebra of $\U_q([\ag.\ag])$ generated by $\mathcal{H}_q^\prime$ and $K_i^{\pm 1},\; i\in I$. Let $B_q^{(d,\prime)}$ be the subalgebra of $\U_q([\ag,\ag])$ with generated by $\U_q^+$ and $\mathcal{H}_q^{(\prime,d)}$.
Let $M^{(\prime,q)}_{\zeta,a}$ be as in equation \ref{b} a $\mathcal{H}_q^\prime$ module. Given any $\lambda\in \dot{P}$, we make $M^{(\prime,q)}_{\zeta,a}$ a  module with $K_i$ action is given by $q^{\lambda(h_i)}$ and extend it to $B_q^{(d,\prime)}$ by trivial action of $\U_q^+$. Then similar to $\qag$, we get a parabolic induction functor depending on $\lambda$, let us denote the image of $M_{\zeta,a}^\prime$ under this functor by 
$\widetilde{M}_{(\zeta,\lambda)}^{\prime,q}$.

\begin{theorem}\label{qafder}
 Assume $(\zeta, a,\lambda )\in (Hom_{\text{alg}}(\mathcal{H}_q^\prime(+), \C(q^{1/2})), \C^\times, \dot{P})$, then $\widetilde{M}_{(\zeta,\lambda)}^{\prime,q}$ is irreducible module over $\U_q([\ag.\ag])$ 
\end{theorem}
\begin{proof}
   We can  construct $\A$-form of $\U_q[\ag,\ag]$ and define $\A$-form of $\widetilde{M}_{(\zeta,\lambda)}^{\prime,q}$. Denote it by ${\widetilde{M}_{(\zeta,\lambda)_{\A}}^{\prime,q}}$ and let $\overline{\widetilde{M}_{(\zeta,\lambda)}^{\prime,q}}_{\A}$ be its classical limits. Any proper submodule of $\widetilde{M}_{(\zeta,\lambda)}^{\prime,q}$ will have a nonzero $\A$-form and the corresponding classical limit will be a proper submodule of  $\overline{\widetilde{M}_{(\zeta,\lambda)}^{\prime,q}}_{\A}$, which contradicts Theorem \ref{da}.
\end{proof}
\section{Proof of the conjecture for $ \widehat{M}_{\zeta,\lambda}^q$.}
In this section, we focus on the irreducibility of the induced module $\widehat{M_{\zeta,\lambda}^q}$ where the derivation $D$ acts freely. As noted in Remark \ref{rmk1}, the classical limit techniques utilized in Section \ref{aform} are not applicable here due to the nature of the $D$ action. Consequently, we provide a direct proof for the case where $\qag$ is the quantum affine algebra of type $A_1^{(1)}$. 

\medskip

The proof strategy relies on a detailed analysis of the commutation relations between Drinfel'd generators and the imaginary PBW basis. We first establish several technical lemmas regarding the adjoint action of the generators $\psi_r$ and $\phi_{-r}$
 on the root vectors. These identities allow us to show that any non-zero submodule must eventually contain an element of $M_{\zeta,a}^{q}$ and eventually it contains the highest \emph{imaginary Whittaker vector}.
 

\begin{lemma}[\cite{B94}, Prop. 3.10, Lemma 3.15]\label{psi}
    We set $a_1=q^2\gamma^{\frac{1}{2}},\; b_1=q^2\gamma^{-\frac{1}{2}}$ and let $r>0$ and $m\in \Z$. Then we have the following equalities:
    \[[\psi_r, x_m^-]=-\gamma^{\frac{1}{2}}[2]\Bigl (a_1^{r-1}x_{m+r}^{-}+\sum_{k=1}^{r-1}a_1^{k-1}(q-q^{-1})\psi_{r-k}x^-_{m+k}\Bigl)\]
     \[[\psi_r, x_m^+]=\gamma^{-\frac{1}{2}}[2]\Bigl(b_1^{r-1}x_{m+r}^{+}+\sum_{k=1}^{r-1}b_1^{k-1}(q-q^{-1})x^+_{m+k}\psi_{r-k}\Bigl).\]
    
\end{lemma}
The  identity in Lemma \ref{psi} describe the action of positive generators $\psi_r$ on the real root vectors. By applying the algebra anti-automorphism $\Omega$ we obtain the corresponding relations for the negative generators $\phi_{-r}$. This symmetry allows us to describe the adjoint action of the negative part of  $\phi_{-r}$ as follows

 
\begin{lemma}

Setting $c=a_1^{-1},\; d=b_1^{-1}$, with $a_1,b_1$ as above, we have 
\[[\phi_{-r}, x_{-m}^+]=\gamma^{-\frac{1}{2}}[2](c^{r-1}x^{+}_{-m-r}-\sum_{k=1}^{r-1}c^{k-1}(q-q^{-1})x^{+}_{-m-k}\phi_{k-r}),\]

\[[\phi_{-r},x_m^-]=-\gamma^{\frac{1}{2}}[2](d^{r-1}x_{m-r}^{-1}-\sum_{k=1}^{r-1}d^{(k-1)}(q-q^{-1})\phi_{k-r}x_{m-k}^{-}).\]
     
\end{lemma}
In the above Lemma, the first commutator is a linear combination of monomials with right orders as in the imaginary PBW basis (Theorem \ref{IPBW}), but the second commutator is not in the right order. In the next proposition, we will present this commutator in the right order.
For this, we first define the formal sum operator for any two numbers $r,j\in \N$ with $j<r$: 
\[\mathbb{S}_{j}^r=\sum_{k_1=1}^{r}\sum_{k_2=1}^{r-k_1}\sum_{k_3=1}^{r-(k_1+k_2)}\dots \sum_{k_j=1}^{r-\sum_{e=1}^{j-1}k_e}.\]
Rewriting the second commutator as a linear combination of monomials with the right orders of the imaginary PBW basis we have the following lemma: 

\begin{lemma}\label{phic}
    Set $t=\gamma^{\frac{1}{2}}[2](q-q^{-1})$ and take $r>0$ and $m\in \Z$. Then we have the following:
     \[[\phi_{-r},x_m^-]
 =-(\sum_{j=1}^{r}\frac{t^j}{q-q^{-1}}d^{r-j}\s^{r-1}_{j-1}(1))x_{m-r}^{-}+\sum_{j=1}^{r-1}t^j\s^{r-1}_j(d^{k_1+\dots k_j-{r+j}}x^{-}_{m-k_1-\cdots -k_j}\phi_{k_1+k_2\cdots +k_j-r}).\]
\end{lemma}
Let us call $r$ the length of $\phi_{-r}$.
This lemma shows that the commutators will give  new $\phi$'s component with a smaller length than $\phi_{-r}$.

From now on we will use the following convention: if 
some expression contains a word with a hat, then this word is removed from the expression. 
\begin{proposition}\label{qprop}
  Let $s\in \N_{>0}$ with $m_1,\cdots ,m_s\in \Z$ and $k\in \Z$ such that $k+m_j<0$ for all $j=1,2,\dots, s$. Then we have
\[x_k^+x_{m_1}^{-}\cdots x_{m_s}^{-}=x_{m_1}^{-}\cdots x^-_{m_s}x_k^+-\sum_{j=1}^s\frac{1}{q-q{-1}}\gamma^{\frac{m_j-k}{2}}x_{m_1}^-\cdots \hat{x}_{m_j}^-\cdots x_{m_s}^{-}\phi_{m_j+k}\]
\[-\sum_{j=1}^s\sum_{r=j+1}^sx_{m_1}^-\cdots \hat{x}_{m_j}^{-}\dots [\phi_{m_j+k},x_{m_r}^{-}]\cdots x_{m_s}^{-}.\]
\end{proposition}
\begin{proof}
    A direct consequence of the continuous use of commutators. 
\end{proof}
Note that in the last sum we will have more monomials when we write in the right order of the PBW monomials. Hence, applying Lemma \ref{phic} we get more monomials, but if we have any new $\phi$'s in those sums, then their lengths must be smaller than the maximum length of $\phi_{m_j+k}$ for $1\leq j\leq s$. \\
\medskip
Given any $t\in \N_{>0}$, we define the set $\mathbb{M}_t:=\{(m_1,\cdots, m_t)\in \Z^t: m_1\leq m_2\leq \dots \leq m_t\}$.\\

\textbf{Proof of the conjecture for $A_1^{(1)}$:}\label{Proof}
We see that the module is a $K$-weight module over $\U_q(A_1^{(1)})$ with the weight set $q^{\lambda(K)-2\N}$. Hence by \cite[Proposition 3.2.1]{HK}, we see that any submodule is also a $K$-weight module. So, it is sufficient to prove that any submodule  contains an element of a $K$-weight $q^{\lambda(K)}$. 
Any element of a submodule of weight $q^{\lambda(K)-2t}$ looks as follows 
\[v=\sum_{\mathbf{m}=(m_1,\dots ,m_t)\in  {I}_v}x^{-}_{m_1} \cdots {x_{m_t}}^{-}\otimes v_{ \mathbf{m}},\]
where $I_v$ is a finite subset of $\mathbb{M}_t$. Given any $\mathbf{m}=(m_1,\dots, m_t)\in \mathbb{M}_t$ there exists $n_{\mathbf{m}}\in \N$ such that $v_{\mathbf{m}}\in \C(q^{\frac{1}{2}})[h_{r}:-n_m\leq r<0]\otimes \C(q^{\frac{1}{2}})[D^{\pm1}].$ Now we choose $k\in \Z_{<0}$ such that $k+m_j<-N$, with $N=\text{Max}\{n_{\mathbf{m}}:\mathbf{m}\in I_v\}$, this will ensure us that given any $\mathbf{m}\in I_v$ and $m_j$ being $j$-th co-ordinate of $\mathbf{m}$, $\phi_{k+m_j}$ will act as left multiplication.  
Now we get

\[x^+_k.v=-\frac{1}{q-q^-}\sum_{\mathbf{m}\in I_v}\sum_{j=1}^tq^{(\frac{m_j-k}{2})l} x_{m_1}^{-}\cdots \hat{x}_{m_j}^-\cdots x_{m_t}^-\phi_{k+m_j}\otimes v_{\mathbf{m}}-\]
\[-\sum_{\mathbf{m}\in I_v}\sum_{j=1}^t\sum_{r=j+1}^tx_{m_1}^{-}\cdots \hat{x}_{m_j}^-\cdots [\phi_{m_j+k},x_{m_r}^-]\cdots x_{m_s}^{-}\otimes v_{\mathbf{m}}.\]
It is easy to see that the first sum is nonzero and $x_{k}^+.v$ is a nonzero weight vector of $K$-weight $q^{\lambda(K)-2(t-1)}$. Hence continuing the process we get the desired result.

\section{Acknowledgments}
The second author would like thank Luan Bezerra and Xingpeng Liu for some insightful discussion about quantum groups. 
\noindent V. Futorny is partially supported by the NSF of China (12350710787 and 12350710178).


\begin{thebibliography}{100}{
    \bibitem[ALZ16]{ALZ16} D. Adamovic, R. Lu, K. Zhao, \emph{Whittaker modules for the affine Lie algebra $A_1^{(1)}$}, Adv. Math, {\bf 289}, 2016, 438-479.
    \bibitem[BM11]{BM11} P. Batra, V. Mazorchuk, \emph{Blocks and modules for Whittaker pairs,}  J. Pure. Appl. Algebra {\bf 215 (7)} (2011) 1552-1568.
    \bibitem[B94]{B94} J. Beck, \emph{Braid group action and quantum affine algebras,} Comm. Math. Phys., 165(3) (1994) 555-568.
    \bibitem[BO09]{BO09} G. Benkart, M. Ondrus, {\sl Whittaker modules for generalized Weyl algebras}, Represent. Theory 13 (2009) 141-164.
    \bibitem[BK78]{BK78} Kostant Bertram, \emph{On Whittaker vectors and representation theory},  Invent Math 48, 101–184 (1978).
    \bibitem[BCFK22]{BCFK22} L. Bezerra, L. Calixto, V. Futorny, I. Kashuba, \emph{Representations of affine Lie superalgebras and their quantization in type A}, J. Algebra 611 (2022), 320-340.
    \bibitem[C08]{C08} K. Christodoulopoulou,  \emph{Whittaker modules for Heisenberg algebras and imaginary Whittaker modules for affine Lie algebras,}  J. Algebra 320 (2008), no. 7, 2871–2890.
    \bibitem[CFM15]{CFM15} B. Cox, V. Futorny, K. Misra, \emph{An imaginary PBW basis for quantum affine algebras of type 1}, J. Pure Appl. Algebra {\bf 2019} (2015), no. 1, 83-100.
    \bibitem[CF23]{CF23} M. C. Cardoso, V. Futorny, \emph{Affine Lie algebra representations induced from Whittaker modules}, Proc. Amer. Math. Soc. 151 (2023), no. 3, 1041–1053.
    \bibitem[CC21]{CH21} C. Chen, {\sl Whittaker modules for classical Lie superalgebras,} Commun. Math. Phys. 388 (1) (2021) 351-383.
    \bibitem[CLLW24]{CLLW24} H. Chen, G. Lin, Z. Li, L. Wang \emph{Classical Whittaker modules for the affine Kac-Moody algebras $A_N^{(1)}$}, Adv. Math. {\bf 454} (2024), P no. 109874.
    \bibitem[CFKM97]{CFKM97} B. Cox, V. Futorny, S. J. Kang, D. Melville, \emph{Quantum deformations of imaginary Verma modules} Proceedings of the London Mathematical Society Volume 74.
     \bibitem[CJ20]{CJ20} X. Chen, C. Jiang, {\sl Whittaker modules for the twisted affine Nappi-Witten Lie algebra $\hat{H}_4[\tau]$}, J. Algebra 546 (2020) 37-61.
    
    \bibitem[DFD09]{DFD09} I Dimitrov, V. Futorny, D Grantcharov, \emph{Parabolic sets of roots}, Contemp. Math., 499 AMS, Prov, RI, 2009, 61-73. 
    \bibitem[D85]{D85} V. G. Drifeld, \emph{Hopf algebras and quantum Yang-Baxter equation}, Dokl. Akad. Nauk SSSR, 283(5) (1985) 1060-1064.
    \bibitem[FHW15]{FHW15} V. Futorny, J. T. Hartwig, E.A. Wilson, \emph{Quantum affine modules for non-twisted affine Kac-Moody algebras}, Proc. Amer. Math. Soc. 143(2015), n0. 12, 5159-5171.

    \bibitem[F90]{F90} V. M. Futorny, \emph{Parabolic partitions of root systems and corresponding representations
of the affine Lie algebras}, Akad. Nauk Ukrain. SSR Inst. Mat. Preprint, (8):30–39, 1990.

\bibitem[F92]{F92} V. M. Futorny. \emph{The parabolic subsets of root system and corresponding representations of affine Lie algebras}, In Proceedings of the International Conference on Algebra, Part
2 (Novosibirsk, 1989), volume 131 of Contemp. Math., pages 45–52, Providence, RI,
1992. Amer. Math. Soc.

\bibitem[FGXZ25]{FGXZ25} V. Futorny, X. Guo, Y. Xue K. Zhao, \emph{Smooth representations of affine Kac-Moody algebras} Adv. Math. 481 (2025), Paper No. 110559, 34 pp.
    
    \bibitem[FT25]{FS25} V. Futorny, S. Tantubay, \emph{Whittaker modules for W type Cartan Lie superalgebras} 	arXiv:2511.17995.
\bibitem[JK85]{JK85}H. P. Jacobsen, V. G. Kac, \emph{A new class of unitarizable highest weight representations of infinite-dimensional Lie algebras} In Nonlinear equations in classical and quantum field theory (Meudon/Paris, 1983/1984), pages 1–20. Springer, Berlin, 1985.
   \bibitem[JK90]{JK90} H. P. Jacobsen, V. G. Kac, \emph{A new class of unitarizable highest weight representations of infinite-dimensional Lie algebras}. II. J. Funct. Anal., 82(1):69–90, 1989}.

   
   
    \bibitem[HK]{HK} J. Hong, S. J. Kang, \emph{Introduction to Quantum Groups and Crystal Bases}, Graduate Studies in Mathematics, Volume 42. 
    \bibitem[Kac]{Kac} Victor G. Kac, \emph{Infinite-Dimensional Lie Algebras,} third edition, Cambridge University Press, Cambridge, 1990.
     
    \bibitem[LL24]{LL24} G Lin, Z. Li, \emph{Classical Whittaker modules for the Classical affine Kac-Moody algebras}, J. Algebra, {\bf 644} (2024), 23-63.
    \bibitem[L88]{L88} G. Lusztig, \emph{Quantum deformations of certain simple modules over enveloping algebras}, Adv. Math. {\bf 70} (1988), no-2, 237-249.
     \bibitem[L93]{L93} G. Lusztig, \emph{Introduction to Quantum Groups}. Boston: Birkh\"{a}user/Springer, New York, 2010, reprint of 1994 edition.
     \bibitem[LPX19]{LPX19} D. Liu, Y. Pei, L. Xia, {Whittaker modules for the super-Virasoro algebras} J. Algebra Appl. 18 (2019) 1950211.
      \bibitem[LWZ10]{LWZ10} D. Liu, Y. Wu, L. Zhu, {\sl Whittaker modules for the twisted Heisenberg-Virasoro algebra} J. Math. Phys. 51 (2010) 023524.
     \bibitem[O05]{MO05} M. Ondrus, \emph{Whittaker Modules over $U_q(\mathfrak{sl}(2))$,} J. Algebra, {\bf 289}, 192-213 (2005).
     \bibitem[OW09]{OW09} M. Ondrus, E. Wiesner, {\sl Whittaker Modules for the Virasoro Algebra} J. Algebra Appl. 8 (2009) 363-377. 
     \bibitem[B81]{REB81} R. E. Block, \emph{The irreducible representations of $\mathfrak{sl}(2)$ and of the Weyl Algebra}, Adv. Math. 39 69-110 (1981).
     \bibitem[S00]{S00}  A. Sevostyanov, {\sl Quantum deformation of Whittaker modules and the Toda Lattice} Duke Math. J. 105 (2000) 211-238.
      \bibitem[XGZ21]{XGZ21}  L. Xia, X. Guo, J. Zhang, {\sl Classification on irreducible Whittaker modules over quantum group $U_q(\mathfrak{sl}_3, \Lambda)$} Front. Math. China 16(4) (2021) 1089-1097.
      \bibitem[ZL22]{ZL22} Y. Zhao, Genqiang Liu, {\sl Whittaker category for the Lie algebra of polynomial vector fields} J. Algebra {\bf 605} (2022) 74-88.      
    \end{thebibliography}
	\end{document}